\documentclass[12pt]{amsart}

\usepackage{amssymb,latexsym}
\usepackage{enumerate}
\usepackage{hyperref}
\usepackage{fullpage}
\usepackage{soul}

\usepackage{fixme}
\fxsetup{
    status=draft,
    author=,
    layout=margin,
    theme=color
}

\makeatletter \@namedef{subjclassname@2010}{%
  \textup{2010} Mathematics Subject Classification}
\makeatother

\newcounter{thm} \numberwithin{thm}{section}
\newtheorem{Theorem}[thm]{Theorem}

\newtheorem{Lemma}[thm]{Lemma}
\newtheorem{Corollary}[thm]{Corollary}

\newtheorem*{Claim}{Claim}

\renewcommand{\ln}[0]{\operatorname{\ell n}}

\usepackage{tikz}
\usetikzlibrary{decorations.pathreplacing}
\tikzset{mybrace/.style={decoration={brace,raise=1.8mm},decorate}}
\tikzset{mybracedown/.style={decoration={brace,mirror,raise=1.8mm},decorate}}
\usetikzlibrary{math}
\usetikzlibrary{patterns}
\usetikzlibrary{matrix,backgrounds}

\author[O. Roche-Newton]{Oliver Roche-Newton} \address{Johannes Kepler Universit\"{a}t\\
Linz, Austria}
\email{o.rochenewton@gmail.com}

\author[D. Zhelezov]{Dmitrii Zhelezov} \address{Johann Radon Institute for Computational and Applied Mathematics\\
Linz, Austria}
\email{dzhelezov@gmail.com}

\date{}

\begin{document}

\baselineskip=17pt

\title{Convexity, Elementary Methods, and Distances}

\date{}
\maketitle

\begin{abstract} This paper considers an extremal version of the Erd\H{os} distinct distances problem. For a point set $P \subset \mathbb R^d$, let $\Delta(P)$ denote the set of all Euclidean distances determined by $P$. Our main result is the following: if $\Delta(A^d) \ll |A|^2$ and $d \geq 5$, then there exists $A' \subset A$ with $|A'| \geq |A|/2$ such that $|A'-A'| \ll |A| \log |A|$.

This is one part of a more general result, which says that, if the growth of $|\Delta(A^d)|$ is restricted, it must be the case that $A$ has some additive structure. More specifically, for any two integers $k,n$, we have the following information: if
\[
| \Delta(A^{2k+3})| \leq |A|^n    
\]
then there exists $A' \subset A$ with $|A'| \geq |A|/2$ and
\[
| kA'- kA'| \leq k^2|A|^{2n-3}\log|A|.
\]
These results are higher dimensional analogues of a result of Hanson \cite{H}, who considered the two-dimensional case.

\end{abstract}





\section{Introduction}

Given a finite\footnote{Henceforth, all sets introduced are assumed to be finite.} point set $P \subset \mathbb R^n$, let $\Delta(P)$ denote the set of distances determined by $P$, that is,
\[
\Delta(P):= \{ \| p - q \| :p,q \in P \},
\]
where $\| \cdot \|$ is the Euclidean norm. The Erd\H{o}s distinct distance problem asks for the best possible lower bound for the size of $\Delta(P)$. In a celebrated work of Guth and Katz \cite{GK}, it was proven that
\begin{equation} \label{GKbound}
|\Delta(P)| \gg \frac{|P|}{ \log |P|}
\end{equation}
holds for all $P \subset \mathbb R^2$, thus resolving the Erd\H{o}s distinct distance problem in the plane, up to a logarithmic factor. The higher dimensional version of Erd\H{o}s's question remains open, where is it is conjectured that, for any $P \subset \mathbb R^d$,
\begin{equation} \label{Econj}
|\Delta (P)| \gg |P|^{\frac{2}{d}}.
\end{equation}

If one makes the additional restriction that the point set $P$ is a Cartesian product, this problem starts to resemble a question about expansion in the spirit of the sum-product problem. For instance, for the case when $P=A \times A$, the Guth-Katz bound \eqref{GKbound} yields
\begin{equation} \label{GKbound2}
|(A-A)^2+(A-A)^2| \gg \frac{|A|^2}{ \log |A|}.
\end{equation}
Because $|\Delta(A^d)| \geq |\Delta(A^2)|$, the bound \eqref{GKbound2} immediately implies the conjectured bound \eqref{Econj}, up to a logarithmic factor, for the Cartesian product case $P=A^d$.

One can also consider structural analogues of these questions, where the goal is to determine which sets can attain the minimal or close to the minimal number of distances. It is expected that the extremal sets have some kind of lattice-like structure. There are several conjectures making this idea more precise, but very little is known concretely. For example, Erd\H{o}s conjectured \cite{E} that, if a set of points $P$ defines $O(|P|/ \sqrt{\log |P|})$ distinct distances, there must exist a line $\ell$ such that $|\ell \cap P| \gg |P|^{1/2}$. This conjecture remains wide open, and more generally speaking, very little progress has been made towards proving that point sets determining few distances must have a special structure.

One notable exception is a paper of Hanson \cite{H} which made the first progress on the structural version of the distance problem in the case when $P=A \times A$. Hanson proved that, if $P=A \times A$ determines close to the minimum number of distances, then $A$ must exhibit some additive structure. To be precise, it was proven in \cite{H}, using a quite elementary argument, that
\[
|(A-A)^2+(A-A)^2| \gg |A-A||A|^{1/8}.
\]
In particular, this shows that
\[
|\Delta (A \times A)| \ll |A|^2 \implies |A-A| \ll |A|^{2-\frac{1}{8}}.
\]
A further quantitative improvement to this result was established by Pohoata \cite{P}. One may hope that an optimal result of this kind takes the form
\[
    |\Delta (A \times A)| \ll |A|^2 \implies  \, \forall \, \epsilon>0, \, |A-A| \ll |A|^{1+\epsilon}.
\]
The main goal of this note is to prove such a result in higher dimensions. We will prove the following.

\begin{Theorem} \label{thm:mainB}
Let $A \subset \mathbb R$. Then there exists $A' \subset A$ such that $|A'| \geq |A|/2$
\[
|\Delta(A^5)| \gg \frac{|A|^{3/2}|A'-A'|^{1/2}}{ ( \log |A|)^{1/2}}.
\]
In particular,
\[
|\Delta(A^5)| \ll |A|^2 \implies |A'-A'| \ll |A| \log |A|.
\]
\end{Theorem}

One can further apply the Freiman-Ruzsa theorem with the quantitative bound of Sanders \cite{S} and verify that a large portion of $A$ may be covered by a generalized arithmetic progression (GAP) of low rank\footnote{The theorem of Sanders implies the existences of a \textit{convex coset progression}, which can be extended to a GAP with negligible losses using a John-type theorem of Tao and Vu \cite{TV}.}. Namely, 
\[
|\Delta(A^5)| \ll |A|^2 
\]
implies that there is a subset $A' \subset A$ of size $\Omega(|A|\exp(-\log^{3+o(1)} \log |A|))$ contained in a GAP of rank  $O(\log^{3+o(1)} \log |A|)$ and of size $O(|A|\exp(\log^{3+o(1)} \log |A|))$.

Theorem \ref{thm:mainB} is derived from a more general result concerning the relationship between the size of the distance set and the difference set.

\begin{Theorem} \label{thm:main}
Let $A \subset \mathbb R$ and $k\in \mathbb N$. Then there exists $A' \subset A$ such that $|A'| \geq |A|/2$
\[
|\Delta(A^{2k+3})| \gg \frac{|A|^{3/2}|k A'-k A'|^{1/2}}{k^{1/2}(\log|A|)^{1/2}}.
\]
\end{Theorem}

The proof of Theorem \ref{thm:main} actually gives us some very precise information about the set $A'$; it consists of the first $|A|/2$ elements of $A$. Theorem \ref{thm:mainB} follows from Theorem \ref{thm:main} by setting $k=1$.

One may think of Theorem \ref{thm:main} as saying the following: if the growth of the distance set $\Delta(A^{2k+3})$ is restricted, it must be the case that the original one-dimensional set $A \subset \mathbb R$ contains additive structure. For instance, Theorem \ref{thm:main} immediately gives the following consequence.
\begin{Corollary}
Let $k,n \in \mathbb N$ and suppose that $\Delta(A^{2k+3}) \leq |A|^n $. Then there exists $A' \subset A$ such that $|A'| \geq |A|/2$ and 
\[
|kA'-kA'| \ll |A|^{2n-3}k^2 \log |A|.
\]
\end{Corollary}

The proof of Theorem \ref{thm:main} makes use of an elementary ``squeezing" argument that has its origins in \cite{RSSS}, and has played a prominent role in recent progress on the sum-product problem, see \cite{B}, \cite{HRNR}, \cite{HRNS} and \cite{M}. 

\subsection{Notation}

Throughout this paper, the notation  $X\gg Y$, $Y \ll X,$ $X=\Omega(Y)$, and $Y=O(X)$ are all equivalent and mean that $X\geq cY$ for some absolute constant $c>0$. $X \approx Y$ and $X=\Theta (Y)$ denote that both $X \gg Y$ and $X \ll Y$ hold. $X \gg_a Y$ means that the implied constant is no longer absolute, but depends on $a$.

The notation $kA$ is used for the iterated sum set, and so 
\[
kA= A+ \dots +A= \{a_1 + \dots + a_k : a_1, \dots, a_k \in A \}.
\]
We use curly brackets to distinguish a set of dilates from an iterated sum set, and so we write $\{k \}A$ for the set $\{ka: a \in A \}$.

\section{Proofs of the main results}

The main goal of this section, and the rest of the paper, is to prove Theorem \ref{thm:main}. All of the other results stated in the introduction follow immediately from this.

We begin with a very simple lemma which gives control of the range for iterated difference sets.

\begin{Lemma} \label{lem:differences}
Let $k \geq 1$ be an integer. If $A \subset  (0,t)$ then $kA-kA \subset (-kt,kt)$.
\end{Lemma}

\begin{proof}
Let $M= \max A$ and $m= \min A$, so $M <t$ and $m>0$. The largest element of $kA-kA$ is $kM-km <kt$. The smallest element of $kA-kA$ is $km-kM >-kt$.

\end{proof}

The key ideas for proving the main results of this paper are contained in the following result, and its proof. The main innovation here is that we can make use of repeated differences of differences when applying the squeezing argument, since the range of these iterated difference sets can be conveniently controlled.

\begin{Lemma} \label{lem:main1}
Suppose that $f:\mathbb R \rightarrow \mathbb R$ is a strictly convex function and that $A=\{a_1<a_2<\dots < a_N \}$ is a set of real numbers, where $N$ is even. Let $k \geq 1$ be an integer. Let $a,a' \in A$ be two elements which satisfy
\[
a'-a= \min \{ a_{j+1} - a_j : j \in [N-1] \}.
\]
Define
\[
D:=\{f(a'+a_{j})-f(a+a_{j}) : j \leq N/2 \}.
\]
Then
\[
 |(k+1)f(a+A) + kf(a'+A) - kf(a+A)-kf(a'+A)| \gg \frac{N|kD-kD|}{k}.
 \]

\end{Lemma}

\begin{proof}
Note that the intervals 
\begin{equation} \label{ints}
(a+a_{j}, a'+a_{j}), \,\,\,\, j = 1,2,\dots  N
\end{equation}
all have equal length $a'-a$ and do not overlap. It then follows from the strict convexity of $f$ that the intervals
\begin{equation} \label{increasing}
(f(a+a_{j}), f(a'+a_{j})), \,\,\,\, j =1,2, \dots, N
\end{equation}
have strictly increasing lengths as $j$ increases, and do not overlap. We use this information to prove the following claim.

\begin{Claim}
Define
\[
t:=f(a'+a_{ \frac{N}{2}})-f(a+a_{ \frac{N}{2} }).
\]
Then, for any integer $ \frac{N}{2k} < \ell \leq \frac{N}{k}-1$, the interval
\[
(f(a+a_{k\ell}), f(a+ a_{k(\ell+1)}))
\]
has length at least $kt$.
\end{Claim}

\begin{proof} The interval $(f(a+a_{k\ell}), f(a+ a_{k(\ell+1)}))$ contains the following $k$ disjoint intervals
\[
(f(a+a_{k\ell}), f(a'+ a_{k\ell})), (f(a+a_{k\ell+1}), f(a'+ a_{k\ell+1})) , \dots ,(f(a+a_{k\ell+(k-1)}), f(a'+ a_{k\ell+(k-1)})).
\]
Since the intervals listed in \eqref{increasing} are strictly increasing in length and $k \ell  > N/2$, each of these intervals has length at least $t$, and the total length of the intervals is at least $kt$.
\end{proof}

Recall that
\[
D:=\{f(a'+a_{j})-f(a+a_{j}) : j \leq N/2 \}.
\]
Once again, the elements of $D$ are strictly increasing with $j$, and so the largest element in $D$ is
\[
 f(a'+a_{ \frac{N}{2}})-f(a+a_{ \frac{N}{2} })=t.
\]
Therefore, by Lemma \ref{lem:differences},
\[
kD - kD \subset (-kt, kt).
\]
We will only use the non-negative elements of this difference set, i.e. the set
\[
(kD - kD)_+:= (kD - kD) \cap [0, kt).
\]
Since $kD - kD$ is a difference set it is symmetric, and it follows that
\[
|(kD - kD)_+| \geq \frac{|kD - kD|}{2}.
\]
Now define
\[ 
A_1:=\left \{ a_{k\ell } :  \frac{N}{2k} < \ell \leq \frac{N}{k}-1\right \}
\]
and consider the sum set
\[
f(a+A_1) + (kD - kD)_+ = \bigcup_{ \frac{N}{2k} < \ell \leq \frac{N}{k}-1} f(a+a_{k\ell})+(kD - kD)_+.
\]
Since $(kD - kD)_+ \subset [0, kt)$ and the interval 
\[
(f(a+a_{k\ell}), f(a+ a_{k(\ell+1)}))
\]
has length at least $kt$ (from the claim), it follows that the sets $f(a+a_{k\ell})+(kD - kD)_+$ are pairwise disjoint as $ \ell$ varies. Therefore, we have
\begin{align*}
 |f(a+A) + kD - kD| &\geq  |f(a+A_1) + (kD - kD)_+| 
 \\ & \geq  \sum_{\frac{N}{2k} < \ell \leq \frac{N}{k}-1} | (kD - kD)_+|
 \\& \gg \frac{N|kD-kD|}{k}.
\end{align*}
Since $f(a+A) + kD - kD \subset (k+1)f(a+A) + kf(a'+A) - kf(a+A)-kf(a'+A)$, the proof is complete.

\end{proof}

A small modification of the proof of Lemma \ref{lem:main1} gives the following statement, where a shift by an element of $-A$ is instead used. This is particularly relevant in the context of applications bounding the size of a set of distances given by a Cartesian product $A^d$.

\begin{Lemma} \label{lem:maindiff}
Suppose that $f:\mathbb R \rightarrow \mathbb R$ is a strictly convex function and that $A=\{a_1<a_2<\dots < a_N \}$ is a set of real numbers, with $N$ even. Let $k \geq 1$ be an integer. Let $a,a' \in A$ be two elements which satisfy
\[
a'-a= \min \{ a_{j+1} - a_j : j \in [N-1] \}.
\]
Define
\[
D:=\{f(a_{j}-a')-f(a_{j}-a) : j \leq N/2 \}.
\]
Then
\[
 |(k+1)f(A-a') + kf(A-a) - kf(A-a')-kf(A-a)| \gg \frac{N|kD-kD|}{k}.
 \]

\end{Lemma}

\begin{proof}
The proof is essentially the same as that of Lemma \ref{lem:main1}, except for a small modification at the outset. Instead of the intervals given in \eqref{ints}, we consider the intervals.
\[
(a_{j}-a', a_{j}-a), \,\,\,\, j = 1,2,\dots  N
\]
These intervals are disjoint and have increasing lengths. The remainder of the proof follows the proof of Lemma \ref{lem:main1}, with appropriate modifications.
\end{proof}


Next, we apply Lemma \ref{lem:maindiff} for the strictly convex function $f(x)=x^2$, giving the following result.

\begin{Corollary} \label{cor:quad}
Let $N$ be an even integer, $A=\{a_1<a_2<\dots <a_N \}$ and $A'=\{a_1<a_2<\dots < a_{N/2} \}$. Then for any $k \in \mathbb N$,
\[
|(2k+1)(A-A)^2-2k(A-A)^2| \gg \frac{|A||k A'-k A'|}{k}.
\]
\end{Corollary}

\begin{proof}
Apply Lemma \ref{lem:maindiff} with $f(x)=x^2$ and $k \in \mathbb N$. We obtain the bound
\begin{equation} \label{application}
|(2k+1)(A-A)^2-2k(A-A)^2| \gg \frac{|A||kD-kD|}{k},
\end{equation}
where $D$ is the set defined in the statement of Lemma \ref{lem:maindiff}. We take a closer look at the set $D$ in this case. Note that, for two distinct elements $a,a' \in A$, we have
\begin{align*}
D-D&=\{f(b-a')-f(b-a)-f(c-a')+f(c-a) : b,c \in A' \} 
\\& = \{(b-a')^2-(b-a)^2-(c-a')^2+(c-a)^2 : b,c \in A' \}
\\ &=\{-2(a'b-ab-a'c+ac)) : b,c \in A' \}
\\ &=\{-2(a'-a)\}(A'-A').
\end{align*}
Therefore,
\[
|kD-kD|=|k(D-D)|=|k(A'-A')|=|kA'-kA'|.
\]
It then follows from \eqref{application} that
\[ 
|(2k+1)(A-A)^2-2k(A-A)^2| 
\gg \frac{|A||kD-kD|}{k}  = \frac{|A||k A'-k A'|}{k}.
\]

\end{proof}

We are now ready to prove Theorem \ref{thm:main}. It is restated below in a slightly different form, recording some additional information about the subset $A' \subset A$.

\begin{Theorem}
Let $N$ be an even integer, let $A= \{a_1 < \dots < a_N \}$ be a set of real numbers and let $A'= \{ a_1< \dots < a_{N/2} \}$. Then
\begin{equation} \label{goal}
|\Delta(A^{2k+3})| \gg \frac{|A|^{3/2}|k A'-k A'|^{1/2}}{k^{1/2}(\log|A|)^{1/2}}.
\end{equation}
\end{Theorem}

\begin{proof}

Apply Corollary \ref{cor:quad}, giving the bound
\begin{equation} \label{step1}
|(2k+1)(A-A)^2-2k(A-A)^2| \gg \frac{|A||kA'-kA'|}{k}.
\end{equation}
Next, we make use of the following fact, which is a consequence of the Ruzsa Triangle Inequality; for any finite sets $X,Y,Z \in \mathbb R$,
\begin{equation} \label{RTI2}
|Y-Z|\leq\frac{|X+Y||X+Z| }{|X|}.
\end{equation}
Applying \eqref{RTI2} with 
\[
X=2(A-A)^2, \,\,\,\,\, Y=(2k+1)(A-A)^2 \,\,\,\, \text{and} \,\,\,\,\, Z=2k(A-A)^2
\]
yields
\begin{align*}
  \frac{|(2k+3)(A-A)^2|^2 }{|2(A-A)^2|} &\geq    \frac{|(2k+3)(A-A)^2||(2k+2)(A-A)^2| }{|2(A-A)^2|}
  \\& \geq |(2k+1)(A-A)^2-2k(A-A)^2| 
  \\ &\gg \frac{|A||kA'-kA'|}{k},
\end{align*}
where the last inequality is \eqref{step1}. Rearranging this inequality and applying \eqref{GKbound2}, we conclude that
\[
|\Delta(A^{2k+3})| = |(2k+3)(A-A)^2| \gg (k^{-1}|A||kA'-kA'| |2(A-A)^2|)^{1/2} \gg \frac{|A|^{3/2}|kA'-kA'|^{1/2}}{k^{1/2}(\log |A|)^{1/2}}.
\]
\end{proof}

Finally, we give an analogue of Corollary \ref{cor:quad} for a different set; an expander determined by taking iterated sums of a product set. Once again, limited growth for this expander implies that the original set has additive structure. In the geometric context, one may see this as an inverse result for Cartesian product sets $A^d$ which determine few distinct dot products.

\begin{Corollary} \label{cor:exp}
Let $N$ be an even integer, $X=\{x_1<x_2<\dots < x_N \}$ and $X'=\{x_1<x_2<\dots < x_{N/2} \}$. Suppose also that $x_1 >0$. Then for any $k \in \mathbb N$,
\[
|(2k+1) XX-2k XX| \gg \frac{|X||k X'-k X'|}{k}.
\]
\end{Corollary}

\begin{proof}
Apply Lemma \ref{lem:main1} with $f(x)=e^x$, $A= \ln X$, and $k \in \mathbb N$. We can calculate that
\[
|D-D|=|X'-X'|,
\]
where $D$ is the set defined in Lemma \ref{lem:main1}. Therefore, Lemma \ref{lem:main1} gives
\[
|(2k+1) XX- 2k XX | \gg \frac{|X||kD-kD|}{k}
= \frac{|X||k X'-kX'|}{k}.
\]

\end{proof}

In fact, we have been rather careless in this proof. If we instead paid more attention to the distributive property, we would obtain the bound
\[
|\{x\}((k+1) X- k X) + \{x'\}(k X- k X)| \gg \frac{|X||kX'-kX'|}{k}
\]
for some $x,x' \in X$.

\section*{Acknowledgements}

Oliver Roche-Newton was supported by the Austrian Science Fund FWF Project P 34180.

\end{document}